\documentclass[10pt,reqno]{amsart}

\usepackage{amsmath,amsthm,amssymb,comment,fullpage}
\usepackage{braket}
\usepackage{mathtools}

\usepackage{caption}
\usepackage{times}
\usepackage[T1]{fontenc}
\usepackage{mathrsfs}
\usepackage{latexsym}
\usepackage[dvips]{graphics}
\usepackage{epsfig}
\usepackage{amsmath,amsfonts,amsthm,amssymb,amscd}
\input amssym.def
\input amssym.tex
\usepackage{color}
\usepackage{hyperref}
\usepackage{url}
\newcommand{\bburl}[1]{\textcolor{blue}{\url{#1}}}

\usepackage{tikz}
\usepackage{tkz-tab}
\usepackage{tkz-graph}
\usetikzlibrary{shapes.geometric,positioning}

\newcommand{\burl}[1]{\textcolor{blue}{\url{#1}}}

\numberwithin{equation}{section}

\newtheorem{thm}{Theorem}[section]

\theoremstyle{plain}

\newtheorem{corollary}[thm]{Corollary}
\newtheorem{definition}[thm]{Definition}
\newtheorem{example}[thm]{Example}
\newtheorem{lemma}[thm]{Lemma}

\newtheorem{theorem}[thm]{Theorem}

\newtheorem{remark}[thm]{Remark}

\newcommand\be{\begin{equation}}
\newcommand\ee{\end{equation}}
\newcommand\bee{\begin{equation*}}
\newcommand\eee{\end{equation*}}
\newcommand\bea{\begin{eqnarray}}
\newcommand\eea{\end{eqnarray}}
\newcommand\beae{\begin{eqnarray*}}
\newcommand\eeae{\end{eqnarray*}}
\newcommand\bi{\begin{itemize}}
\newcommand\ei{\end{itemize}}
\newcommand\ben{\begin{enumerate}}
\newcommand\een{\end{enumerate}}
\newcommand\bc{\begin{center}}
\newcommand\ec{\end{center}}
\newcommand\ba{\begin{array}}
\newcommand\ea{\end{array}}

\newcommand\frakfamily{\usefont{U}{yfrak}{m}{n}}
\DeclareTextFontCommand{\textfrak}{\frakfamily}

\newcommand{\hr}[1]{\href{#1}{\url{#1}}}

\newcommand{\PP}[1]{\mathbb{P}[#1]}
\newcommand{\E}[1]{\mathbb{E}[#1]}
\newcommand{\V}[1]{\text{{\rm Var}}[#1]}

\title{ON THE ASYMPTOTIC BEHAVIOR OF VARIANCE OF PLRS DECOMPOSITIONS}

\author{Steven J. Miller}
\email{\textcolor{blue}{\href{mailto:sjm1@williams.edu}{sjm1@williams.edu}},  \textcolor{blue}{\href{Steven.Miller.MC.96@aya.yale.edu}{Steven.Miller.MC.96@aya.yale.edu}}}
\address{Department of Mathematics and Statistics, Williams College, Williamstown, MA 01267}

\author{Dawn Nelson}
\email{\textcolor{blue}{\href{mailto:dnelson1@saintpeters.edu}{dnelson1@saintpeters.edu}}}
\address{Department of Mathematics, Saint Peter's University, Jersey City, NJ 07306}

\author{Zhao Pan}
\email{\textcolor{blue}{\href{mailto:zhaop@andrew.cmu.edu}{zhaop@andrew.cmu.edu}}}
\address{Department of Mathematics, Carnegie Mellon University, Pittsburgh, PA 15213}

\author{Huanzhong Xu}
\email{\textcolor{blue}{\href{mailto:huanzhox@andrew.cmu.edu}{huanzhox@andrew.cmu.edu}}}
\address{Department of Mathematics, Carnegie Mellon University, Pittsburgh, PA 15213}

\thanks{The first named author was partially supported by NSF grants DMS1265673 and DMS1561945 and Carnegie Mellon University. We thank the participants at the 17\textsuperscript{th} International Conference on Fibonacci Numbers and their Applications for helpful discussions.}



\date{\today}

\begin{document}

\begin{abstract}
A positive linear recurrence sequence is of the form $H_{n+1} = c_1 H_n + \cdots + c_L H_{n+1-L}$ with each $c_i \ge 0$ and $c_1 c_L > 0$, with appropriately chosen initial conditions. There is a notion of a legal decomposition (roughly, given a sum of terms in the sequence we cannot use the recurrence relation to reduce it) such that every positive integer has a unique legal decomposition using terms in the sequence; this generalizes the Zeckendorf decomposition, which states any positive integer can be written uniquely as a sum of non-adjacent Fibonacci numbers. Previous work proved not only that a decomposition exists, but that the number of summands $K_n(m)$ in legal decompositions of $m \in [H_n, H_{n+1})$ converges to a Gaussian. Using partial fractions and generating functions it is easy to show the mean and variance grow linearly in $n$: $a n + b + o(1)$ and $C n + d  + o(1)$, respectively; the difficulty is proving $a$ and $C$ are positive. Previous approaches relied on delicate analysis of polynomials related to the generating functions and characteristic polynomials, and is algebraically cumbersome. We introduce new, elementary techniques that bypass these issues. The key insight is to use induction and bootstrap bounds through conditional probability expansions to show the variance is unbounded, and hence $C > 0$ (the mean is handled easily through a simple counting argument).
\end{abstract}

\maketitle

\section{Introduction}

There are many ways to define the Fibonacci numbers. An equivalent approach to the standard recurrence relation, where $F_{n+1} = F_n + F_{n-1}$ and  $F_1 = 1$ and $F_2=2$, is that they are the unique sequence of integers such that every positive number can be written uniquely as a sum of non-adjacent terms. This expansion is called the Zeckendorf decomposition \cite{Ze}, and much is known about it. In particular, the distribution of the number of summands of $m \in [F_n, F_{n+1})$ converges to a Gaussian as $n\to\infty$, with mean and variance growing linearly with $n$. Similar results hold for a large class of sequences which have a notion of legal decomposition leading to unique decomposition; see \cite{Al, BDEMMTTW, CFHMN1, CFHMN2, Day, DDKMMV, DDKMV, DG, GT, GTNP, Ha, Ho, Ke, LT, Len, Lek, KKMW, MW1, MW2, Ste1, Ste2}.

Given a sequence $\{H_n\}$, one can frequently prove that the mean and the variance of the number of summands of  $m \in [H_n, H_{n+1})$ grows linearly with $n$. Explicitly, there are constants $a, b, C$ and $d$ such that the mean is $an+b+o(1)$ and the variance is $Cn + d  + o(1)$. The difficulty is proving that $a$  and $C$ are positive, which is needed for the proofs of Gaussian behavior. Until recently, the only approaches have been technical and involved generating functions, partial fraction expansions and generalized Binet formulas applied to polynomials associated to the characteristic polynomials of the sequence, which have required a lot of work to show the leading terms are positive for such recurrences. The point of this work is to bypass these arguments through elementary counting.  We concentrate on positive linear recurrence sequences (defined below) to highlight the main ideas of the method; with additional work these arguments can be extended to more general sequences (see \cite{CFHMNPX}). In addition to the arguments below, one can also obtain similar results (though not as elementarily) through Markov chains \cite{B-AM} or through an analysis of two dimensional recurrences \cite{LiM}.

\begin{definition}
	\label{defn:goodrecurrencereldef}\label{def:goodrecurrence} A sequence $\{H_n\}_{n=1}^\infty$ of positive integers is a \textbf{Positive Linear Recurrence Sequence (PLRS)} if the following properties hold.

	\begin{enumerate}
	\item \emph{Recurrence relation:} There are non-negative integers $L, c_1, \dots, c_L$\label{c_i} such that \begin{equation*} H_{n+1} \ = \ c_1 H_n + \cdots + c_L H_{n+1-L},\end{equation*} with $L, c_1$ and $c_L$ positive.
	\item \emph{Initial conditions:} $H_1 = 1$, and for $1 \le  n < L$ we have
	\begin{equation*} H_{n+1} \ =\
	c_1 H_n + c_2 H_{n-1} + \cdots + c_n H_{1}+1.
	\end{equation*}
	\end{enumerate}
We define the \textbf{size} of $\{H_n\}$ to be $c_1+\cdots+c_L$ and  the \textbf{length} of $\{H_n\}$ to be $L$.
\end{definition}
	
\begin{definition} Let $\{H_n\}$ be a PLRS. A decomposition $\sum_{i=1}^{m} {a_i H_{m+1-i}}$\label{a_i} of a positive integer $\omega$ (and the sequence $\{a_i\}_{i=1}^{m}$) is \textbf{legal}\label{legal} if $a_1>0$, the other $a_i \ge  0$, and one of the following two conditions holds.
	
	\begin{itemize}
		
		\item Condition 1: We have $m<L$ and $a_i=c_i$ for $1\le  i\le  m$.
		
		\item Condition 2: There exists $s\in\{1,\dots, L\}$ such that
		\begin{equation*}
		a_1\ = \ c_1,\ a_2\ = \ c_2,\ \dots,\ a_{s-1}\ = \ c_{s-1}\ {\rm{and}}\ a_s<c_s,
		\end{equation*}
		and $\{b_i\}_{i=1}^{m-s}$ (with $b_i = a_{s+i}$) is legal.
		
	\end{itemize}
	
	If $\sum_{i=1}^{m} {a_i H_{m+1-i}}$ is a legal decomposition of $\omega$, we define the \textbf{number of summands}\label{summands} (of this decomposition of $\omega$) to be $a_1 + \cdots + a_m$.
	
	Furthermore, we define {two types of blocks}, where a \textbf{block} is a nonempty ordered subset of the coefficients $[a_i, a_{i+1},\ldots,a_{i+j}]$ inclusive:
	\begin{itemize}
		\item a \textbf{Type 1 block} corresponds to Condition 1, and has length $m < L$ and size $a_i+ \cdots +a_{i+m-1}$,
		\item a \textbf{Type 2 block} corresponds to Condition 2, and has length $s\le  L$ and size $a_i+ \cdots +a_{i+s-1}$.
	\end{itemize}
\end{definition}

\begin{remark}
A Type 2 block has three key properties.
	\begin{itemize}
		\item A legal decomposition of $\omega$ stays legal if  a Type 2 block is inserted (between Type 1 and/or 2 blocks) or removed and indices are shifted appropriately.
		\item If we know the size of a Type 2 block, the block's content and its length are uniquely determined. So we can define a \textbf{length function} $\ell(t)$ to be the length of a Type 2 block with size $t$.
		\item A Type 2 block always has nonnegative size and strictly positive length. Specifically, consider a Type 2 block with size 0. Then, in Condition 2, we always have $a_1 = 0 < c_1$, so $s = 1$. Thus a Type 2 block with size 0 has length 1. In other words, $\ell(0) = 1$ holds for all PLRS.
	\end{itemize}
	
	If a legal decomposition contains a Type 1 block, then it must be the last block. Thus any legal decomposition contains at most one Type 1 block. A Type 1 block, according to Condition 1, always has positive size and positive length.
\end{remark}

The following two examples clarify the above.

\begin{example}
	\textbf{The Fibonacci Sequence} (size 2 and length 2).\\
	Type 1 block: [1].\\
	Type 2 blocks: [0], [1 0].\\
	An example of a legal decomposition: $F_5 + F_3 + F_1$ with block representation: [1 0] [1 0] [1].\\
	After removing the second to last block, the new block representation is [1 0] [1].\\
	The resulting legal decomposition is $F_3 + F_1$.\\	
\end{example}

\begin{example}\textbf{PLRS sequence $H_n = 2H_{n-1} + 2H_{n-2} + 0 + 2H_{n-4}$} (size 6 and length 4).\\
	Type 1 blocks: [2], [2 2], [2 2 0].\\
	Type 2 blocks: [0], [1], [2 0], [2 1], [2 2 0 0], [2 2 0 1].\\
	An example of a legal decomposition: $H_7 + 2H_4 + H_1$ with block representation: [1] [0] [0] [2 0] [0] [1].\\
	After removing the second to last block, the new block representation is [1] [0] [0] [2 0] [1].\\
	The resulting legal decomposition is $H_6 + 2H_3 + H_1$.\\
\end{example}

Before we state our main result we first set some notation.

\begin{definition}
Let $\{H_n\}$ be a Positive Linear Recurrence Sequence. For each $n$, let the discrete outcome space $\Omega_n$ be the set of legal decompositions of integers in $[H_n, H_{n+1})$. By the Generalized Zeckendorf Theorem (see for example \cite{MW2}) every integer has a unique legal decomposition, so $|\Omega_n| = H_{n+1} - H_n$. Define the probability measure on subsets of $\Omega_n$ by \begin{equation*}\mathbb{P}_n(A) \ = \ \sum_{\omega \in A \atop \omega \in \Omega_n} \frac1{H_{n+1}-H_n}, \ \ \  A \subset \Omega_n; \end{equation*} thus each of the $H_{n+1}-H_n$ legal decompositions is weighted equally. We define the random variable $K_n$ by setting $K_n(\omega)$ equal to the number of summands of $\omega \in \Omega_n$. When $n > 2L$ (so there are at least three blocks) we define the random variable $Z_n$ by setting $Z_n(\omega)$ equal to the size of the second to last block of $\omega \in \Omega_n$. Note that the second to last block must be a Type 2 block. Finally, we define the random variable $L_n$ by setting $L_n(\omega)$ equal to the length of the second to last block of $\omega \in \Omega_n$; i.e., $L_n(\omega) = \ell(Z_n(\omega))$.
\end{definition}

As remarked above, previous work has shown that $\E{K_n} = an + b + f(n)$ where $a > 0$ and $f(n) = o(1)$; this can be proved through very simple counting arguments (see \cite{CFHMNPX}). While it is also known that $\V{K_n} = Cn + d + o(1)$, previous approaches could not easily show $C \neq 0$. We elementarily prove $C > 0$ by giving a positive lower bound $c$ for it.

\begin{theorem}\label{thm:main} Let $\{H_n\}$ be a positive linear recurrence sequence with size $S$ and length $L$. Then there is a $c > 0$ such that $\V{K_n}  \ge  cn$ for all $n >  L$.
\end{theorem}

We sketch the proof. We can remove the second to last block of a legal decomposition to get a shorter legal decomposition, forming relations between longer legal decompositions and shorter legal decompositions. We then use strong induction and conditional probabilities to prove the theorem.

\begin{remark}
As it is known that $\V{K_n} = Cn + d + o(1)$, to prove that $C > 0$ it would suffice to show  $\lim_{n\to\infty} {\rm Var}{[K_n]}$ diverges to infinity. Unfortunately the only elementary proofs we could find of this also establish the correct growth rate; we would be very interested in seeing an approach that yielded (for example) ${\rm Var}[K_n] \gg \log n$ (which would then immediately improve to implying $C>0$).
\end{remark}


\section{Lemmas derived from Expectation}

We first determine a relationship between $K_n$ and $Z_n$. Then, with the help of $\E{K_n} = an + b + f(n)$, we explain how to explicitly determine the positive lower bound $c$.

\begin{lemma}
	Let $n > 2L$. For all $0 \le  t < S$, we define $S_t := \{\omega \in \Omega_n|Z_n(\omega) = t\}$, and $h_t(\omega)$ to be the decomposition after removing the second to last block of $\omega$. (When we remove the second to last block with size $t$, we completely remove that block from $\omega$ and shift all the indices to the left of that block  by $\ell(t)$.) When we remove the second to last block (a Type 2 block) from $\omega$, then $h_t(\omega)$ is legal and $h_t$ is a bijection between $S_t$ and $\Omega_{n-\ell(t)}$.
\end{lemma}

\begin{proof}
Let $\omega \in \Omega_n$ be arbitrary and consider $h_t(\omega)$. Since the block we remove has size $t$ and thus length $\ell(t)$,  $h_t(\omega)$ must be in $\Omega_{n-\ell(t)}$.

Next, consider $\omega, \omega' \in S_t$, such that $h_t(\omega) = h_t(\omega')$. As the size determines the composition for Type 2 blocks, we are removing the same block at the same position for $\omega, \omega'$. This implies $\omega = \omega'$.

Finally, for any $\omega \in \Omega_{n-\ell(t)}$, if we insert the size $t$ type 2 block before its last block, we get a legal decomposition in $\Omega_n$. Thus $h_t$ is surjective.

Therefore, $h_t$ is a bijection between $S_t$ and $\Omega_{n-\ell(t)}$
\end{proof}

\begin{corollary} We have
	\begin{equation*}\begin{split}
		\PP{Z_n = t} \ = \ \frac{|S_t|}{|\Omega_n|} \ =\  \frac{|\Omega_{n-\ell(t)}|}{|\Omega_n|} \ =\  \frac{H_{n - \ell(t) + 1} - H_{n - \ell(t)}}{H_{n+1} - H_n}.
	\end{split}\end{equation*}
\end{corollary}

\begin{remark}
As \begin{equation}\PP{Z_n = 0}\ \ge\  \PP{Z_n = 1} \ge  \dots \ge  \PP{Z_n = S-1}\end{equation} and the sum of these $S$ terms is 1, we have \begin{equation}\label{8}\PP{Z_n = 0}\ \ge\  \frac{1}{S},\end{equation} (which is the consequence we need below).
\end{remark}

For an arbitrary $\omega \in S_t$, the second to last block has size $Z_n = t$, and the remaining blocks form a legal decomposition in $\Omega_{n-\ell(t)}$ with size $K_{n-\ell(t)}(h_t(\omega))$, so $K_n(\omega) = K_{n-\ell(t)}(h_t(\omega)) + t$. Since $h$ is a bijection, we have the following two equations:
		\begin{equation}\label{1}
			\begin{split}
				\E{K_n|Z_n = t} &\ = \  \E{K_{n-\ell(t)} + t}\\
				&\ = \  a(n-\ell(t)) + b + f(n-\ell(t)) + t,
			\end{split}
		\end{equation}
	and
		\begin{equation}\label{2}
			\begin{split}
				\E{K_n^2|Z_n = t} &\ = \  \E{(K_{n-\ell(t)} + t)^2}\\
				&\ = \  \E{K_{n-\ell(t)}^2 + 2t{K_{n-\ell(t)}} + t^2}\\
				&\ = \  \E{K_{n-\ell(t)}^2} + 2t\E{K_{n-\ell(t)}} + t^2\\
				&\ = \  \E{K_{n-\ell(t)}^2} + 2t[a(n-\ell(t)) + b + f(n-\ell(t))] + t^2.
			\end{split}
		\end{equation}
Furthermore, by $\eqref{1}$ we have
		\begin{equation}\label{3}
			\begin{split}
			\E{K_n} &\ = \  \sum\limits_{t = 0}^{S-1} \PP{Z_n = t}\cdot \E{K_n|Z_n = t}\\
			&\ = \  \sum\limits_{t = 0}^{S-1} \PP{Z_n = t}\cdot [a(n-\ell(t)) + b + f(n-\ell(t)) + t]\\
			&\ = \  an + b + \sum\limits_{t  = 0} ^ {S-1} \PP{Z_n = t}\cdot [t + f(n-\ell(t)) - a\ell(t)]\\
			&\ = \  an + b + f(n),\\
			\end{split}
		\end{equation}
where the last equality comes from the definition of $f(n)$.

If we set $Y_n(\omega) := Z_n(\omega) + f(n-L_n(\omega)) - aL_n(\omega)$, then we have
\begin{equation}\label{4}
	\E{Y_n} \ =\  \sum\limits_{t  = 0} ^ {S-1} \PP{Z_n = t}\cdot [t + f(n-\ell(t)) - a\ell(t)] \ = \ f(n).
\end{equation}

Now that we have $\E{Y_n}$, we use it to estimate $\V{Y_n}$.

\begin{lemma}\label{lem:varyn} For $n$ sufficiently large we have \begin{equation}\label{9}
		\V{Y_n}\ > \ \frac{a^2}{2S}.
	\end{equation}
\end{lemma}

\begin{proof}
	First, for all $n > 2L$ we have
	\begin{equation*}
		\begin{split}
		\V{Y_n} &\ = \  \E{Y_n^2} - \left(\E{Y_n}\right)^2\\
		&\ = \  \left(\E{(Z_n - aL_n + f(n - L_n))^2}\right) - \left(f(n)\right)^2 \\
		&\ = \  \left(\E{(Z_n - aL_n)^2}+\E{2(Z_n - aL_n)\cdot f(n - L_n)} + \E{f(n - L_n)^2}\right) - \left(f(n)\right)^2.
		\end{split}
	\end{equation*}
	Note that $Z_n - aL_n$ is bounded since $ -aL \le  Z_n - aL_n \le  S$ for all $n > 2L$. Also we know $f(n) = o(1)$, so $f(n-L_n) = o(1)$ since $L_n \le  L$. Hence the following three limits are all zero:
	\begin{equation} \lim\limits_{n \to \infty}\E{2(Z_n - aL_n)\cdot f(n - L_n)}\  = \ \lim\limits_{n \to \infty}\E{f(n - L_n)^2}\ = \ \lim\limits_{n \to \infty}  \left(f(n)\right)^2 \ =\ 0.\end{equation}
	
Further, we know \begin{equation} \V{Y_n} - \E{(Z_n - aL_n)^2}\ =\ \E{2(Z_n - aL_n)\cdot f(n - L_n)} + \E{f(n - L_n)^2} - \left(f(n)\right)^2,\end{equation} so
	\begin{equation}\label{5}
		\lim\limits_{n \to \infty}\left(\V{Y_n} - \E{(Z_n - aL_n)^2}\right)\ =\ 0.
	\end{equation}
	On the other hand, for all $n > 2L$ we have
	\begin{equation}\label{6}\begin{split}
		\E{(Z_n - aL_n)^2} &\ = \  \sum\limits_{t = 0}^{S - 1} \PP{Z_n = t} \cdot \left(t - a\ell(t)\right)^2\\
		&\ \ge \ \PP{Z_n = 0} \cdot \left(0 - a\ell(0)\right)^2\\
		&\ \ge \ \frac{a^2}{S},
	\end{split}\end{equation}
	where the last inequality follows from  $\eqref{8}$.
	
	By $\eqref{5}$, we know there must exist $N > 2L$ such that for all $n > N$, $|\V{Y_n} -  \E{(Z_n - aL_n)^2}| < \frac{a^2}{2S}$, so $\V{Y_n} -  \E{(Z_n - aL_n)^2} > -\frac{a^2}{2S}$. Then, by $\eqref{6}$, we get $\V{Y_n} > \frac{a^2}{2S}$ for all $n > N$.
\end{proof}

Finally, we choose $c$. Let
\begin{equation} c\ =\ \text{min}\left\{\frac{\V{K_{L+1}}}{L+1},\ \frac{\V{K_{L+2}}}{L+2},\ \dots,\ \frac{\V{K_{N}}}{N},\ \frac{a^2}{2SL}\right\},\end{equation}
Where $N$ is as determined in Lemma \ref{lem:varyn}. For all $n >  L$, $H_{n + 1} - H_n > 1$, so there are at least two integers in $[H_n, H_{n+1})$. Since the legal decomposition of $H_n$ has only one summand while that of $H_n + 1$ has two summands, $\V{K_n}$ is nonzero when $n >  L$. Hence, $c > 0$. In the next section we show $\V{K_n} \ge  cn$ for all $n >  L$.


\section{A lower bound for the Variance}

We prove Theorem \ref{thm:main} by strong induction. While the algebra is long, the main idea is easily stated: we condition based on how many summands are in the second to last block, which must be a type 2 block, and then use conditional probability arguments (inputting results for the mean and smaller cases) to compute the desired quantities.

\begin{proof}
The base cases $n = L+1, L+2, \dots, N$  are automatically true by the way we choose $c$. Hence, we only need to consider the cases when $n > N$. In the induction hypothesis, we assume $\V{K_r} \ge  cr$ for $L <  r < n$. In the inductive step, we prove $\V{K_n} \ge  cn$ where $n > N$.

	For $L <  r < n$, we have $\V{K_r} \ge  cr$ and $\E{K_r} = ar + b + f(r)$, hence
	\begin{equation}\label{7}
		\begin{split}
			\E{K_r^2} &\ =\ \V{K_r} + \left(\E{K_r}\right)^2\\
			&\ \ge\  cr + (ar + b + f(r))^2\\
			&\ = \  cr + a^2r^2 + b^2 + (f(r))^2 + 2arb + 2arf(r) + 2bf(r).
		\end{split}
	\end{equation}

	By $\eqref{2}$, we have
	\begin{equation*}\begin{split}
		\E{K_n^2} &\ = \  \sum\limits_{t=0}^{S-1}\PP{Z_n = t}\cdot\E{K_n^2|Z_n = t}\\
		&\ = \  \sum\limits_{t=0}^{S-1}\PP{Z_n = t} \cdot \left(\E{K_{n-\ell(t)}^2} + 2t[a(n-\ell(t)) + b + f(n-\ell(t))] + t^2\right).
	\end{split}\end{equation*}
	Note we only need to consider $n > N > 2L$, so $n > n-\ell(t) \ge  n - L >  L$ for all $0 \le t \le  S-1$. Hence, by $\eqref{7}$,
	\begin{equation*}\begin{split}
		\E{K_{n-\ell(t)}^2} &\ \ge\  c(n-\ell(t)) + a^2(n-\ell(t))^2 + b^2 + [f(n-\ell(t))]^2 + 2a(n-\ell(t))b\\
		&\ \quad\ + 2a(n-\ell(t))f(n-\ell(t)) + 2bf(n-\ell(t)).
	\end{split}\end{equation*}
	After we replace $\E{K_{n-\ell(t)}^2}$ in the conditional expectation $\E{K_n^2|Z_n = t}$ with this lower bound, any term either does not depend on $t$ or can be combined with other terms to form $(t+f(n-\ell(t))-a\ell(t))$. The final equation will then have two parts, one of which does not depend on $t$, while the other can be written in the form of $Z_n + f(n - L_n) - aL_n$, which is exactly $Y_n$. We find
	\begin{equation*}\begin{split}
		\E{K_n^2} &\ \ge \ \sum\limits_{t=0}^{S-1}\PP{Z_n = t} \cdot \Bigg[\left(c(n-\ell(t)) + a^2(n-\ell(t))^2 + b^2 + [f(n-\ell(t))]^2 + 2a(n-\ell(t))b\right.\\
		&\ \quad\ \left. + 2a(n-\ell(t))f(n-\ell(t)) + 2bf(n-\ell(t))\right)+ 2t[a(n-\ell(t)) + b + f(n-\ell(t))] + t^2\Bigg]\\
		&\ = \  \sum\limits_{t=0}^{S-1}\PP{Z_n = t} \cdot \Bigg[\left(c(n-\ell(t)) + a^2(n-\ell(t))^2 + b^2 + [f(n-\ell(t))]^2 + 2a(n-\ell(t))b\right. \\
		&\ \quad\ \left. + 2a(n-\ell(t))f(n-\ell(t)) + 2bf(n-l(
t))\right)+ \left(2tan - 2ta\ell(t) + 2tb + 2tf(n-\ell(t))\right) + t^2\Bigg]\\
		&\ = \  (an+b)^2 + cn + \sum\limits_{t = 0}^{S-1}\PP{Z_n = t}\cdot \Bigg[-c\ell(t) - 2a^2n\ell(t) + a^2(\ell(t))^2 + [f(n-\ell(t))]^2 - 2a\ell(t)b\\
		&\ \ \ \ + \  2anf(n-\ell(t)) - 2a\ell(t)f(n-\ell(t)) + 2bf(n-\ell(t)) + 2tan - 2ta\ell(t) + 2tb + 2tf(n-\ell(t)) + t^2  \Bigg]\\
		&\ = \  (an+b)^2 + cn + \sum\limits_{t = 0}^{S-1}\PP{Z_n  = t}\cdot \Bigg[\left(a^2(\ell(t))^2 + [f(n-\ell(t))]^2 + t^2 - 2a\ell(t)f(n-\ell(t)) - 2ta\ell(t)\right. \\
		& \ \ \ + \  \left. 2tf(n-\ell(t)) \right) + 2an\left(t + f(n-\ell(t)) - a\ell(t)\right) + 2b\left(t+f(n-\ell(t))-a\ell(t)\right) - c\ell(t)\Bigg]\\
		&\ = \  (an+b)^2 + cn + \sum\limits_{t = 0}^{S-1}\PP{Z_n = t}\cdot\Bigg[\left(t+f(n-\ell(t))-a\ell(t)\right)^2 \\
		&\ \ \ \ + \  2(an+b)\left(t+f(n-\ell(t))-a\ell(t)\right) - c\ell(t)\Bigg]\\
		&\ = \  (an+b)^2 + cn + \sum\limits_{t = 0}^{S-1}\PP{Z_n = t}\cdot\left(t+f(n-\ell(t))-a\ell(t)\right)^2\\
		&\ \ \ \ + \  2(an+b)\sum\limits_{t=0}^{S-1}\PP{Z_n = t}\cdot\left(t+f(n-\ell(t)) -a\ell(t)\right) - c\sum\limits_{t=0}^{S-1}\PP{Z_n = t}\cdot \ell(t)\\
		&\ = \  (an+b)^2 + cn + \E{(Z_n+f(n-L_n) -aL_n)^2} + 2(an+b)f(n) - c\E{L_n},
	\end{split}\end{equation*}
	where the last equality comes from $\eqref{4}$.
	
	We already know $\left(\E{K_n}\right)^2 = (an+b+f(n))^2 = (an+b)^2 + 2(an+b)f(n) + (f(n))^2$, hence
	\begin{equation*}\begin{split}
		\V{K_n} - cn &\ =\ \E{K_n^2} - \left(\E{K_n}\right)^2 - cn\\
		&\ \ge\  \E{(Z_n+f(n-L_n) -aL_n)^2} - c\E{L_n}-(f(n))^2\\
		&\ = \  \E{Y_n^2} - c\E{L_n} - \left(\E{Y_n}\right)^2\\
		&\ = \  \V{Y_n} - c\E{L_n}\\
		&\ \ge \ \V{Y_n} - cL\\
		&\ \ge \ 0,
	\end{split}\end{equation*}
	where the last inequality comes from our definition of $c$ and $\eqref{9}$.
	
	Therefore, $\V{K_n} \ge  cn$ for all $n >  L$. In other words, if $\V{K_n} = Cn + d + o(1)$, then $C \ge  c > 0$.
\end{proof}


\ \\


\begin{thebibliography}{BBBGILMT11} 

\bibitem[Al]{Al}
H. Alpert,  \emph{Differences of multiple Fibonacci numbers}, Integers: Electronic Journal of Combinatorial Number Theory  \textbf{9} (2009), 745--749.


\bibitem[B-AM]{B-AM}
I. Ben-Ari  and S. J. Miller, \emph{A Probabilistic Approach to Generalized Zeckendorf Decompositions}, \textbf{30} (2016), no. 2, 1302--1332.



\bibitem[BDEMMTTW]{BDEMMTTW}
A. Best, P. Dynes, X. Edelsbrunner, B. McDonald, S. J. Miller, K. Tor, C. Turnage-Butterbaugh, M. Weinstein, \emph{Gaussian Distribution of Number Summands in Zeckendorf Decompositions in Small Intervals}, Fibonacci Quarterly. (52 (2014), no. 5, 47--53).



\bibitem[CFHMN1]{CFHMN1}
M. Catral, P. Ford, P. E. Harris, S. J. Miller, and D. Nelson, \emph{Generalizing Zeckendorf's Theorem: The Kentucky Sequence}, Fibonacci Quarterly. (52 (2014), no. 5, 68--90).

\bibitem[CFHMN2]{CFHMN2}
M. Catral, P. Ford, P. E. Harris, S. J. Miller, and D. Nelson, \emph{Legal Decompositions Arising from Non-positive Linear Recurrences}, preprint 2016  (expanded arXiv version). \bburl{http://arxiv.org/pdf/1606.09312}.

\bibitem[CFHMNPX]{CFHMNPX}
M. Catral, P. Ford,  P. E. Harris, S. J. Miller, D. Nelson, Z. Pan and H. Xu, \emph{New Behavior in Legal Decompositions Arising from Non-positive Linear Recurrences}, preprint 2016 (expanded arXiv version). \bburl{http://arxiv.org/pdf/1606.09309}.

\bibitem[Day]{Day}
D. E. Daykin,  \emph{Representation of Natural Numbers as Sums of Generalized Fibonacci Numbers}, J. London Mathematical Society \textbf{35} (1960), 143--160.

\bibitem[DDKMMV]{DDKMMV}
P. Demontigny, T. Do, A. Kulkarni, S. J. Miller, D. Moon and U. Varma, \emph{Generalizing Zeckendorf's Theorem to $f$-decompositions}, Journal of Number Theory \textbf{141} (2014), 136--158.

\bibitem[DDKMV]{DDKMV}
P. Demontigny, T. Do, A. Kulkarni, S. J. Miller and U. Varma, \emph{A Generalization of Fibonacci Far-Difference Representations and Gaussian Behavior}, Fibonacci Quarterly \textbf{52} (2014), no. 3, 247--273.




\bibitem[DG]{DG}
M. Drmota and J. Gajdosik, \emph{The distribution of the sum-of-digits function},
J. Th\'eor. Nombr\'es Bordeaux \textbf{10} (1998), no. 1, 17--32.







\bibitem[GT]{GT}
P. J. Grabner and R. F. Tichy, \emph{Contributions to digit expansions with respect to linear recurrences}, J. Number Theory \textbf{36} (1990), no. 2, 160--169.

\bibitem[GTNP]{GTNP}
P. J. Grabner, R. F. Tichy, I. Nemes, and A. Peth\"o, \emph{Generalized Zeckendorf expansions}, Appl. Math. Lett. \textbf{7} (1994), no. 2, 25--28.

\bibitem[Ha]{Ha} N. Hamlin,  \emph{Representing Positive Integers as a Sum of Linear Recurrence Sequences}, Fibonacci Quarterly \textbf{50} (2012), no. 2, 99--105.

\bibitem[Ho]{Ho} V. E. Hoggatt,  \emph{Generalized Zeckendorf theorem}, Fibonacci Quarterly \textbf{10} (1972), no. 1 (special issue on representations), pages 89--93.

\bibitem[Ke]{Ke} T. J. Keller,  \emph{Generalizations of Zeckendorf's theorem}, Fibonacci Quarterly \textbf{10} (1972), no. 1 (special issue on representations), pages 95--102.

\bibitem[LT]{LT}
M. Lamberger and J. M. Thuswaldner, \emph{Distribution properties of digital expansions
arising from linear recurrences}, Math. Slovaca \textbf{53} (2003), no. 1, 1--20.

\bibitem[Len]{Len} T. Lengyel, \emph{A Counting Based Proof of the Generalized Zeckendorf's Theorem}, Fibonacci Quarterly \textbf{44} (2006), no. 4, 324--325.

\bibitem[Lek]{Lek} C. G. Lekkerkerker,  \emph{Voorstelling van natuurlyke getallen door een som van getallen van Fibonacci}, Simon Stevin \textbf{29} (1951-1952), 190--195.

\bibitem[LiM]{LiM}
R. Li and S. J. Miller, \emph{Central Limit Theorems for gaps of Generalized Zeckendorf Decompositions}, preprint 2016. \bburl{http://arxiv.org/pdf/1606.08110}.



\bibitem[KKMW]{KKMW} M. Kolo$\breve{{\rm g}}$lu, G. Kopp, S. J. Miller and Y. Wang,  \emph{On the number of summands in Zeckendorf decompositions}, Fibonacci Quarterly \textbf{49} (2011), no. 2, 116--130.




\bibitem[MW1]{MW1}
S. J. Miller and Y. Wang,  \emph{From Fibonacci numbers to Central Limit Type Theorems}, Journal of Combinatorial Theory, Series A \textbf{119} (2012), no. 7, 1398--1413.

\bibitem[MW2]{MW2}
S. J. Miller and Y. Wang, \emph{Gaussian Behavior in Generalized Zeckendorf Decompositions}, Combinatorial and Additive Number Theory, CANT 2011 and 2012 (Melvyn B. Nathanson, editor), Springer Proceedings in Mathematics \& Statistics (2014), 159--173.





\bibitem[Ste1]{Ste1}
W. Steiner, \emph{Parry expansions of polynomial sequences}, Integers \textbf{2} (2002), Paper A14.

\bibitem[Ste2]{Ste2}
W. Steiner, \emph{The Joint Distribution of Greedy and Lazy Fibonacci Expansions}, Fibonacci Quarterly \textbf{43} (2005), 60--69.

\bibitem[Ze]{Ze} E. Zeckendorf, \emph{Repr\'esentation des nombres naturels par une somme des nombres de Fibonacci ou de nombres de Lucas}, Bulletin de la Soci\'et\'e Royale des Sciences de Li\'ege \textbf{41} (1972), pages 179--182.

\end{thebibliography}
\end{document}